\documentclass{amsart}
\usepackage{latexsym}
\usepackage{amsmath}
\usepackage{amssymb}
\usepackage{amsthm}
\usepackage{amscd}
\usepackage{enumerate} 
\usepackage{amssymb} 
\usepackage{mathrsfs}
\usepackage[all]{xy}
\usepackage{color}
\usepackage{graphicx}
\usepackage{mathtools}
\usepackage{comment}
\usepackage{multirow}
\usepackage{hyperref}
\usepackage[foot]{amsaddr}
\hypersetup{colorlinks=true}

\makeatletter
  
  \@addtoreset{equation}{thm}
  \makeatother

\title{Classification of Du Val del Pezzo surfaces of Picard rank one in characteristic two and three}
\author{\textsuperscript{1}Tatsuro Kawakami}
\email{\textsuperscript{1}tatsurokawakami0@gmail.com}
\address{\textsuperscript{1}Department of Mathematics, Graduate School of Science, Kyoto University, Kyoto 606-8502, Japan}
\author{\textsuperscript{2}Masaru Nagaoka}
\email{\textsuperscript{2}masaru.nagaoka@gakushuin.ac.jp}
\address{\textsuperscript{2}Gakushuin University, 1-5-1 Mejiro, Toshima-ku, Tokyo 171-8588, Japan}

\def\phi{\varphi}
\def\epsilon{\varepsilon}

\def\mapsto{\longmapsto}

\def\Hom{\operatorname{Hom}}
\def\Spec{\operatorname{Spec}}

\def\PGL{\operatorname{PGL}}
\def\MW{\operatorname{MW}}
\def\Dyn{\operatorname{Dyn}}

\newcommand{\Q}{\mathbb{Q}} 
\newcommand{\C}{\mathbb{C}}

\newcommand{\PP}{\mathbb{P}}
\newcommand{\FF}{\mathbb{F}}
\newcommand{\ZZ}{\mathbb{Z}}

\newcommand{\sO}{\mathcal{O}}

\theoremstyle{plain}
\newtheorem{thm}{Theorem}[section] 
\newtheorem{cor}[thm]{Corollary}
\newtheorem{prop}[thm]{Proposition}

\newtheorem{lem}[thm]{Lemma}
\theoremstyle{definition} 
\newtheorem{defn}[thm]{Definition}

\theoremstyle{remark}
\newtheorem{rem}[thm]{Remark}

\newtheorem{defn and notation}[thm]{Definition and Notation}
\newtheorem*{notation}{Notation}


\begin{document}
\tolerance = 9999

\maketitle
\markboth{Tatsuro Kawakami and Masaru Nagaoka}{Du Val del Pezzo surfaces of Picard rank one}

\begin{abstract}
In this paper, we classify Du Val del Pezzo surfaces of Picard rank one in characteristic two and three. 
We also show that if a Du Val del Pezzo surface is Frobenius split, then a general anti-canonical member is smooth.
Furthermore, in characteristic two, it is an ordinary elliptic curve.
\end{abstract}

\tableofcontents

\section{Introduction}
We say that $X$ is a \textit{Du Val del Pezzo surface} if $X$ is a normal projective surface with only Du Val singularities whose anti-canonical divisor is ample.
Du Val del Pezzo surfaces have an important role in algebraic geometry, especially in positive characteristic.
For example, they appear as general fibers of del Pezzo fibrations with terminal singularities in positive characteristic \cite{FS}, although a general fiber is smooth in characteristic zero by the generic smoothness.

Over the field of complex numbers $\C$,
Furushima \cite{Fur} classified the singularities of $X$ of Picard rank $\rho(X)=1$ in terms of corresponding Dynkin diagrams, which we call \textit{the Dynkin type of $X$} in this paper.
For example, we say that $X$ is of type $3A_1+D_4$ if $X$ has three $A_1$-singularities and one $D_4$-singularity. 
In this case, we also write $\Dyn(X) = 3A_1+D_4$ and $X=X(3A_1+D_4)$. 
Later, Ye \cite{Ye} classified isomorphism classes of $X$ over $\C$ of $\rho(X)=1$, and showed that such $X$ is uniquely determined up to isomorphism by $\Dyn(X)$ except for $E_8$, $A_1+E_7$, $A_2+E_6$, and $2D_4$.
Recently, Lacini \cite[Theorem B.7]{Lac} pointed out that the same result as in \cite{Ye} holds in characteristic at least five.

On the other hand, in characteristic two, several Du Val del Pezzo surfaces of Picard rank one have Dynkin types which do not appear in Furushima's list, and we classified such surfaces in \cite[Theorem 1.7 (1)]{KN}.
In particular, the same result as in \cite{Ye} does not hold in characteristic two.
Moreover, in characteristic two or three, Artin \cite{Art} pointed out that Dynkin types do not determine the isomorphism classes of Du Val singularities, and classified isomorphism classes of Du Val singularities via \textit{Artin coindices}.

In this paper, we investigate Du Val del Pezzo surfaces of Picard rank one in characteristic two or three.
Our main result consists of two theorems.
One is Theorem \ref{isom}; combining \cite[Theorem B.7]{Lac} and our study, we obtain the complete classification of Du Val del Pezzo surfaces of Picard rank one in positive characteristic.

\begin{thm}\label{isom}
Let $X$ be a Du Val del Pezzo surface over an algebraically closed field of characteristic $p\geq 0$. 
Suppose that $X$ is singular and the Picard rank of $X$ is one. 
Then the following holds.
\begin{enumerate}
    \item[\textup{(1)}] The Dynkin types of $X$ and the number of the isomorphism classes of the del Pezzo surfaces of the given Dynkin type are listed in Table \ref{Pic1}.
    \item[\textup{(2)}] When $p=2$ $($resp.~$p=3)$, $X$ is uniquely determined up to isomorphism by its Dynkin type with Artin coindices except when its Dynkin type is $D_8$, $2D_4$, $4A_1+D_4$, or $8A_1$ $($resp.~$2D_4$$)$.
\end{enumerate}
    \begin{table}[ht]
\caption{Dynkin types and the number of isomorphism classes of Du Val del Pezzo surfaces of Picard rank one}
  \begin{tabular}{|c|c|c|c|c|c|c|} \hline
\multicolumn{3}{|c||}{Dynkin type}                & \multicolumn{2}{|c|}{$E_8$} & \multicolumn{2}{|c|}{$D_8$}  \\ \hline 
\multicolumn{3}{|c||}{Characteristic}             &$p=2, 3$      & $p \neq 2,3$      & $p=2$     & $p \neq 2$       \\ \hline
\multicolumn{3}{|c||}{No. of isomorphism classes} & $3$          & $2$          & $\infty$  & $1$              \\ \hline\hline
$A_8$ &$A_1+A_7$ & $2A_4$ &$A_1+A_2+A_5$    &$A_3+D_5$ & $4A_2$   & $2A_1+D_6$     \\ \hline
\multicolumn{7}{|c|}{$p \geq 0$}                                                                      \\ \hline
$1$   & $1$    & $1$         &$1$              & $1$      & $1$  & $1$       \\ \hline\hline
 $A_2+E_6$ &$A_1+E_7$         &  $2D_4$  &$2A_1+2A_3$ & $4A_1+D_4$                  &\multicolumn{1}{c|}{$8A_1$}      & $A_7$  \\ \hline
\multicolumn{3}{|c|}{$p\geq 0$}            &$p \neq 2$     & \multicolumn{2}{|c|}{$p=2$}& $p\geq 0$                                                \\ \hline
$2$        &$2$               &     $\infty$   &$1$         & $\infty$                    &\multicolumn{1}{c|}{$\infty$}    & $1$   \\ \hline\hline
 \multicolumn{3}{|c|}{$E_7$}  &\multicolumn{2}{|c|}{$A_1+D_6$} & $A_2+A_5$ & $3A_1+D_4$   \\ \hline
 $p=2$ & $p=3$ & $p\neq 2,3$      &$p=2$ & $p \neq 2$              &\multicolumn{2}{|c|}{$p\geq 0$}                \\ \hline
 $3$   & $2$     & $1$        &$2$    & $1$                    &$1$        &$1$                 \\ \hline\hline
$A_1+2A_3$ &\multicolumn{1}{c|}{$7A_1$} &  \multicolumn{2}{|c|}{$E_6$}&$A_1+A_5$ & \multicolumn{1}{c|}{$3A_2$}& {$2A_1+A_3$} \\ \hline
$p\geq 0$      & \multicolumn{1}{c|}{$p=2$} &  $p=2, 3$ & $p \neq 2,3$         &\multicolumn{2}{|c|}{$p\geq 0$}            & {$p\geq0$}      \\ \hline
$1$        &\multicolumn{1}{c|}{$1$}    & $2$       & $1$             &$1$       & \multicolumn{1}{c|}{$1$}   & {$1$}        \\ \hline\hline
  \multicolumn{2}{|c|}{$D_5$}        & \multicolumn{1}{c|}{$A_4$} & \multicolumn{1}{c|}{$A_1+A_2$}  & $A_1$\\ \cline{1-5}
 $p=2$ & \multicolumn{1}{|c|}{$p\neq 2$} &\multicolumn{1}{c|}{$p\geq 0$}&\multicolumn{1}{c|}{$p\geq 0$}  & $p\geq 0$\\ \cline{1-5}   
  $2$  & \multicolumn{1}{|c|}{$1$}   & \multicolumn{1}{c|}{$1$}& \multicolumn{1}{c|}{$1$}                  &   $1$\\ \cline{1-5}
  \end{tabular}
  \label{Pic1}
\end{table}
\end{thm}
 
\begin{rem}
In \cite[Theorem 1.7 (1)]{KN}, we have shown that $\{7A_1, 8A_1, 4A_1+D_4\}$ is the set of the Dynkin types of Du Val del Pezzo surfaces in characteristic two which are not realizable in characteristic zero.
Theorem \ref{isom} shows that there are also the Dynkin types of Du Val del Pezzo surfaces in characteristic zero which are not realizable in characteristic two; $2A_1+2A_3$ is an example.
\end{rem}

A \textit{Frobenius splitting} (\textit{$F$-splitting}, for short) property is expected to induce similar properties to those of varieties in characteristic zero.
The other main theorem is Theorem \ref{F split Intro}, which supports this expectation. Note that all anti-canonical members of Du Val del Pezzo surfaces can be singular in positive characteristic as we pointed out in \cite[Theorem 1.4]{KN}.
\begin{thm}\label{F split Intro}
Let $X$ be a Du Val del Pezzo surface over an algebraically closed field of characteristic $p>0$.
Suppose that $X$ is $F$-split.
Then a general member of $|-K_X|$ is smooth. 
Moreover, if $p=2$, then a general member of $|-K_X|$ is an ordinary elliptic curve. 
\end{thm}
\begin{rem}
We cannot drop the assumption on characteristic in the latter assertion of Theorem \ref{F split Intro}. Indeed, there exists an $F$-split Du Val del Pezzo surface in $p=3$ such that all smooth anti-canonical members are supersingular elliptic curves.
We refer to Remark \ref{supersing} for the details.
\end{rem}

This paper is structured as follows.
In Section \ref{sec:pre}, we recall some facts on Du Val del Pezzo surfaces, rational genus one fibrations, and Frobenius splittings.

In Section \ref{sec:rho1}, we prove Theorem \ref{isom}.
To show assertion (1), we follow the method used in \cite{Ye}, and use \cite[Theorem 1.4]{KN} and the classification of extremal rational (quasi-)elliptic fibrations \cite{Lang1, Lang2, Ito1, Ito2}.
We also calculate the defining equations of some Du Val del Pezzo surfaces in weighted projective spaces carefully to obtain the assertion (2).

In Section \ref{sec:app}, we prove Theorem \ref{F split Intro}. 
First, we prove the latter assertion by making use of a splitting section of $X$, which is contained in $|-K_X|$ when $p=2$.
After that, we prove the former assertion together with equations of Du Val del Pezzo surfaces which we will determine in Subsection \ref{DVdP} and Section \ref{sec:rho1}.

\begin{notation}
We work over an algebraically closed field $k$ of characteristic $p>0$.
A \textit{variety} means an integral separated scheme of finite type over $k$. 
A \textit{curve} (resp.~a \textit{surface}) means
a variety of dimension one (resp.~two). 
Throughout this paper, we also use the following notation:
\begin{itemize}
\item $E_f$: the reduced exceptional divisor of a birational morphism $f$.
\item $\rho(X)$: the Picard rank of a projective variety $X$.
\item $\MW(X)$: the Mordell-Weil group of a (quasi-)elliptic surface $X$, i.e.,\ the group consisting of sections. 
We refer to \cite[\S 1]{Shioda90} and \cite[\S 2]{Ito1} for the details.
\end{itemize}
\end{notation}

\section{Preliminaries}
\label{sec:pre}

\subsection{Du Val del Pezzo surfaces}\label{DVdP}
In this subsection, we gather basic results on Du Val del Pezzo surfaces.

\begin{defn}\label{GdelPezzo}
Let $X$ be a normal projective surface.
We say that $X$ is a \emph{Du Val del Pezzo surface} if $X$ has only Du Val singularities and $-K_X$ is ample.
We write $\Dyn(X)$ for the formal sum of Dynkin types of singularities on $X$, which we call \textit{the Dynkin type of $X$}.

In characteristic $p>5$, Du Val singularities are taut, i.e., their isomorphism classes are determined by their Dynkin types.
On the other hand, Du Val singularities of type $D_n$ and $E_n$ in $p=2$, $E_n$ in $p=3$, and $E_8$ in $p=5$ are not taut. 
The isomorphism classes of non-taut Du Val singularities are classified and named as $D_n^r$ and $E_n^r$ by Artin \cite[\S 3]{Art}, where $r$ is called the \textit{Artin coindex}. 
We write $\Dyn'(X)$ for the Dynkin type of $X$ with Artin coindices. 
\end{defn}

\begin{lem}\label{basic}
Let $X$ be a Du Val del Pezzo surface of degree $d \coloneqq K_X^2$.
Then the following hold.
\begin{enumerate}\renewcommand{\labelenumi}{$($\textup{\arabic{enumi}}$)$}
\item{A general anti-canonical member is irreducible and reduced. In particular, it is a locally complete intersection curve of arithmetic genus one. Moreover, if $p>3$, then a general anti-canonical member is smooth.}
\item{If $d\geq3$, then $-K_X$ is very ample. }
\item{If $d=4$, then $X$ is isomorphic to a complete intersection of two quadric hypersurfaces in $\PP^4_k$.}
\item{If $d=3$, then $X$ is isomorphic to a cubic surface in $\PP^3_k$.}
\item{If $d=2$, then $|-K_X|$ is basepoint free and $X$ is isomorphic to a weighted hypersurface in $\PP_k(1, 1, 1, 2)$ of degree four.}
\item{If $d=1$, then $|-K_X|$ has the unique base point and $X$ is isomorphic to a weighted hypersurface in $\PP_k(1, 1, 2, 3)$ of degree six.} 
\end{enumerate}
\end{lem}
\begin{proof}
We refer to \cite[Propositions 2.12, 2.14, and Theorem 2.15]{BT} and \cite[Proposition 4.6]{Kaw2} for the proof.
\end{proof}


\subsection{Rational genus one fibrations}

In this subsection, we compile the results on rational extremal elliptic surfaces by Lang \cite{Lang1, Lang2} and rational quasi-elliptic surfaces by Ito \cite{Ito1, Ito2}, which we will use in Section \ref{sec:rho1}.
Since they are not minimal, they have at least one $(-1)$-curve, which is a section.

\begin{thm}[\cite{Lang1, Lang2, Ito3}]\label{extell}
When $p=2$ $(\textit{resp.~}p=3)$, the configurations of singular fibers of extremal rational elliptic surfaces 
and the order of their Mordell-Weil groups are listed in Table \ref{extell2} $($resp.~Table \ref{extell3}$)$ using Kodaira's notation.

Moreover, there is a unique extremal rational elliptic surface with each configuration of singular fibers in the table except for the type I (resp.~type VI).
In this case, there are infinitely many isomorphism classes of extremal rational elliptic surfaces with that configuration of singular fibers.
    \begin{table}[ht]
\caption{The singular fibers and the order of their Mordell-Weil groups of extremal rational ellptic surfaces in $p=2$}
  \begin{tabular}{|c|c|c||c|c|c|} \hline
        &Singular fibers      &$|\mathrm{MW}(X)|$ & & Singular fibers      &$|\mathrm{MW}(X)|$ \\ \hline \hline
       \textup{I}   &$\textup{I}_4^*$                 &$2$  & \textup{VII}  & $\textup{IV}, \textup{IV}^*$            & $3$  \\ \hline   
       \textup{II}  &$\textup{II}^*$                  &$1$  & \textup{VIII} &$\textup{IV}, \textup{I}_2, \textup{I}_6$         & $6$  \\ \hline   
       \textup{III} &$\textup{III}, \textup{I}_8$     &$4$  & \textup{IX}   &$\textup{IV}^*, \textup{I}_1, \textup{I}_3$       & $3$ \\ \hline                    
       \textup{IV}  &$\textup{I}_1^*, \textup{I}_4$   &$4$  & \textup{SI}   &$\textup{I}_9, \textup{I}_1, \textup{I}_1, \textup{I}_1$   & $3$    \\ \hline        
       \textup{V}   &$\textup{III}^*, \textup{I}_2$   &$2$  & \textup{SII}  & $\textup{I}_5, \textup{I}_5, \textup{I}_1, \textup{I}_1$  & $5$     \\ \hline
       \textup{VI}  &$\textup{II}^*, \textup{I}_1$    &$1$  & \textup{SIII} & $\textup{I}_3, \textup{I}_3, \textup{I}_3, \textup{I}_3$  & $9$     \\ \hline
  \end{tabular}
  \label{extell2}
\end{table}
    \begin{table}[ht]
\caption{The singular fibers and the order of their Mordell-Weil groups of extremal rational ellptic surfaces in $p=3$}
  \begin{tabular}{|c|c|c||c|c|c|} \hline
        & Singular fibers &$|\mathrm{MW}(X)|$ &  & Singular fibers &$|\mathrm{MW}(X)|$ \\ \hline \hline
       \textup{I}   & $\textup{II}^*$                              & $1$ & \textup{VIII} & $\textup{I}_1^*, \textup{I}_1, \textup{I}_4$     & $4$ \\ \hline
       \textup{II}  & $\textup{II}, \textup{I}_9$                  & $3$ & \textup{IX}   & $\textup{I}_2^*, \textup{I}_2, \textup{I}_2$     & $4$ \\ \hline
       \textup{III} & $\textup{IV}^*, \textup{I}_3$                & $3$ & \textup{X}    & $\textup{I}_4^*, \textup{I}_1, \textup{I}_1$     & $2$ \\ \hline
       \textup{IV}  & $\textup{II}^*, \textup{I}_1$                & $1$ & \textup{XI}   & $\textup{III}^*, \textup{I}_1, \textup{I}_2$     & $2$ \\ \hline                
       \textup{V}   & $\textup{III}^*, \textup{III}$               & $2$ & \textup{SI}   & $\textup{I}_8, \textup{I}_2, \textup{I}_1, \textup{I}_1$ & $4$ \\ \hline
       \textup{VI}  & $\textup{I}_0^*, \textup{I}_0^*$             & $4$ & \textup{SII}  &$\textup{I}_5, \textup{I}_5, \textup{I}_1, \textup{I}_1$   & $5$ \\ \hline
       \textup{VII} & $\textup{III}, \textup{I}_3, \textup{I}_6$   & $6$ & \textup{SIII} &$\textup{I}_4, \textup{I}_4, \textup{I}_2, \textup{I}_2$   & $8$ \\ \hline
  \end{tabular}
  \label{extell3}
\end{table}

\end{thm}

\begin{thm}[\cite{Ito1, Ito2} and {\cite[Remark 5.23]{KN}}]\label{q-ell}
When $p=2$ $($resp.~$p=3)$, the configurations of reducible fibers of rational quasi-elliptic surfaces 
and the order of their Mordell-Weil groups are listed in Table \ref{q-ell2} $($resp.~Table \ref{q-ell3}$)$ using Kodaira's notation. \begin{table}[ht]
\caption{The singular fibers and the order of their Mordell-Weil groups of rational quasi-ellptic surfaces in $p=2$}
  \begin{tabular}{|c|c|c||c|c|c|} \hline
        &Singular fibers      &$|\mathrm{MW}(X)|$ & & Singular fibers      &$|\mathrm{MW}(X)|$ \\ \hline \hline
       \textup{(a)} &$\textup{II}^*$            &$1$  & \textup{(e)}  &$\textup{I}^*_2, \textup{III}, \textup{III}$       & $4$  \\ \hline   
       \textup{(b)} &$\textup{I}^*_4$          &$2$  & \textup{(f)}  &$\textup{I}^*_0$ and four $\textup{III}$  & $8$  \\ \hline  
       \textup{(c)} &$\textup{III}^*, \textup{III}$     &$2$  & \textup{(g)}  &eight $\textup{III}$             & $16$ \\ \hline       
       \textup{(d)} &$\textup{I}^*_0, \textup{I}^*_0$   &$4$  &      &   &    \\ \hline 
  \end{tabular}
  \label{q-ell2}
\end{table}
    \begin{table}[ht]
\caption{The singular fibers and the order of their Mordell-Weil groups of rational quasi-ellptic surfaces in $p=3$}
  \begin{tabular}{|c|c|c|} \hline
        & Singular fibers &$|\mathrm{MW}(X)|$ \\ \hline \hline
       \textup{(1)} & $\textup{II}^*$          & $1$  \\ \hline
       \textup{(2)} & $\textup{IV}^*, \textup{IV}$      & $3$  \\ \hline
       \textup{(3)} & four $\textup{IV}$       & $9$  \\ \hline
  \end{tabular}
  \label{q-ell3}
\end{table}

Moreover, 
there is a unique rational quasi-elliptic surface with each configuration of reducible fibers in the table except when $p=2$ and the type is one of $(d)$, $(f)$, and $(g)$.
The isomorphism classes of rational quasi-elliptic surfaces of each of types $(d)$ and $(f)$ $(\textit{resp.~type }(g))$ correspond to the closed points of $\mathcal{D}_n / \PGL(n+1, \FF_2)$ with $n=1$ $(\textit{resp.~}n=2)$, where $\mathcal{D}_n \subset \PP^n_k$ is the complement of the union of all the hyperplane sections defined over $\FF_2$.
\end{thm}

\begin{lem}\label{lem:ellconfig}
Let $S$ be an extremal rational elliptic surfaces and $T$ the root lattice associated with the reduced fiber of $S$.
Then the intersection matrix $M$ of all the sections and the irreducible components of reduced fibers on $S$ is uniquely determined by $T$.
\end{lem}

\begin{proof}
By the classification of reducible fibers of elliptic fibrations by Kodaira \cite{Kod}, $T$ determines the intersection matrix of irreducible components of reduced fibers.
The number of sections is also determined by $T$ and sections are disjoint from each other by \cite[Main theorem and Proposition 5.4]{OS}.
Thus it suffices to show that $T$ determines how sections intersect with reducible fibers.

Let $\{F_v\}_{v \in R}$ be the set of reducible fibers of $S$ and fix a zero-section $O$ of $S$.
Then \cite[Theorem 8.6]{Shioda90} shows that 
\begin{align}\label{eq:height}
    1+\delta_{PQ}=\sum_{v \in R} \mathrm{contr}_v(P,Q)
\end{align}
for each two sections $P$ and $Q$, where $\delta_{PQ}$ is the Kronecker delta and $\mathrm{contr}_v(*,*)$ is a rational number determined by the configuration of $O$, $P$, $Q$, and $F_v$ (see [\textit{ibid}, (8.16)] for more detail).
Then an easy computation shows that the equation (\ref{eq:height}) has a unique solution for $\mathrm{contr}_v(P,Q)$ and $M$ is uniquely determined.

We give the precise computation only for the case where $T$ is a root lattice of type $A_8$; the other case are left to the reader.
In this case, $S$ contains exactly three sections $O$, $P_1$, and $P_2$, and one reducible fiber $F_v$, which is of type $\textup{I}_9$.
Take the irreducible decomposition $F_v=\sum_{i=0}^{8} \Theta_{i}$ such that $O=\Theta_{0}$ and $(\Theta_{i} \cdot \Theta_{i+1})=1$ for $0 \leq i \leq 7$.
For $i=1,2$, take $n_i$ such that $(P_i \cdot \Theta_{n_i})=1$.
We may assume that $n_1 \leq n_2$.
Then $\mathrm{contr}_v(P_1,P_2)=n_1(9-n_2)/9$ and $\mathrm{contr}_v(P_i,P_i)=n_i(9-n_i)/9$ for $i=1,2$ by \cite[(8.16)]{Shioda90}.
Now (\ref{eq:height}) gives $(n_1, n_2)=(3,6)$ and hence $M$ is uniquely determined.
\end{proof}


\subsection{Frobenius splittings}
In this subsection, we review the fundamental properties of $F$-split varieties.

\begin{defn}\label{GFS}
Let $X$ be a normal variety. 
We say that $X$ is \textit{$($globally$)$ $F$-split} if the Frobenius map $\mathcal{O}_X \to F_{*}\mathcal{O}_X$ splits as an $\mathcal{O}_X$-module homomorphism. 
We call $\sigma \in \Hom_{\sO_X}(F_{*}\sO_X, \sO_X)\cong H^0(X, \sO_X((1-p)K_X))$ \textit{a splitting section} if $\sigma$ induces a splitting of $\sO_X\to F_{*}\sO_X$. We often call the divisor $\Sigma\in |(1-p)K_X|$ corresponding to $\sigma$ a splitting section of $X$.
\end{defn}

\begin{rem}\label{remF-split}\,
\begin{enumerate}
\item
Let $f\colon X\to Y$ be a projective morphism between normal varieties such that $f_{*}\sO_X=\sO_Y$. If $X$ is $F$-split, then so is $Y$ (see \cite[1.1.9 Lemma]{fbook}).

\item
Let $\pi\colon Y\to X$ be a birational projective morphism between normal $\Q$-Gorenstein varieties such that $f^{*}K_X-K_Y$ is an effective $\Q$-divisor. If $X$ is $F$-split, then so is $Y$ (see \cite[Proposition 1.4]{HWY}). In particular, the minimal resolution of a Du Val del Pezzo surface preserves the $F$-splitting property.

\item
Let $C$ be an elliptic curve. Then $C$ is $F$-split if and only if $C$ is ordinary (see \cite[1.3.9 Remark (ii)]{fbook}).
\end{enumerate}
\end{rem}

The following lemma is an immediate consequence of Fedder's criterion \cite[Proposition 1.7]{Fed}.
\begin{lem}\label{Fedder}
Fix coordinates $[x: y: z: w]$ of $\PP_k(1,1,2,3)$ (resp.~$\PP_k(1,1,1,2), \PP^3_k$) and let $f(x,y,z,w)=0$ be the defining equation of a Du Val del Pezzo surface $X$ of degree one (resp.~two, three).
Then $X$ is $F$-split if and only if $f^{p-1}\not \in (x^p,y^p,z^p,w^p)$.
\end{lem}
\begin{proof}
Let $R\coloneqq k[x,y,z,w]/(f)$.
By Fedder's criterion \cite[Proposition 1.7]{Fed}, $\Spec\, R$ is $F$-split if and only if $f^{p-1}\not \in (x^p,y^p,z^p,w^p)$.
Since $R\cong \bigoplus_{m\geq0} H^0(X, \sO_X(-mK_X))$ and $-K_X$ is ample, it follows that $X$ is $F$-split by \cite[Proposition 4.10]{Smithbook}.
\end{proof}

\begin{lem}[\textup{\cite[Proposition 2.1]{LMP}}]\label{F-split;bl-up}
Let $X$ be a normal $F$-split variety and $Z$ a smooth closed subscheme of codimension $d$ which is contained in the smooth locus of $X$. 
Let $\Sigma \in |(1-p)K_X|$ be a splitting section. 
If $\Sigma$ passes through $Z$ with multiplicity at least $(d-1)(p-1)$, then the blow-up of $X$ along $Z$ is $F$-split.
\end{lem}
\begin{proof}
We refer to the proof of \cite[Proposition 2.1]{LMP} for the details.
\end{proof}

\section{Classification of Du Val del Pezzo surfaces in characteristic two and three}
\label{sec:rho1}

In this section, we prove Theorem \ref{isom}.
When $p=0$ (resp.\ $p > 3$), the assertion (1) has already proven by \cite[Theorem 1.2]{Ye} (resp.\ \cite[Theorem B.7]{Lac}).
For this reason, we assume that $p=2$ or $3$ in this section.
The proof is similar in spirit to \cite{Ye}.
However, we have to follow Ye's method carefully because the classification of extremal rational elliptic surfaces in $p=2$ or $3$ is different from that in $p > 3$ and rational quasi-elliptic surfaces appear in $p=2$ or $3$.
We also investigate $\Dyn'(X)$ of some Du Val del Pezzo surfaces $X$ to get the assertion (2).

\subsection{Reduction to genus one fibrations}

We first recall the relation between Du Val del Pezzo surfaces of Picard rank one and genus one fibrations.

\begin{defn}
For a smooth weak del Pezzo surface $Y$ and the union $D$ of all the $(-2)$-curves in $Y$, 
a curve $E \subset Y$ is called \textit{a nice exceptional curve} (NEC for short) if $E$ is a $(-1)$-curve such that $(E \cdot D)=1$.
\end{defn}

\begin{lem}\label{reduction}
Let $X$ be a Du Val del Pezzo surface of $\rho(X)=1$ and $d=K_X^2 \leq 7$. 
Let $Y \to X$ be the minimal resolution.
Then there are an extremal rational elliptic surface or a rational quasi-elliptic surface $Y_0$ and blow-downs $\{f_i \colon Y_{i-1} \to Y_{i}\}_{1 \leq i \leq d}$ of $(-1)$-curves such that $Y_d=Y$ and $E_{f_i}$ is an NEC for $2 \leq i \leq d$.
\end{lem}

\begin{proof}
Since a general anti-canonical member is a curve of arithmetic genus one by Theorem \ref{basic} (1), 
the same proof as in 
\cite[Theorem B.6]{Lac} remains valid for this case after admitting for $Y_0$ to be a rational quasi-elliptic surface.
\end{proof}

\begin{rem}\label{rem:matrix}
We follow the notation of Lemma \ref{reduction}.
Note that we can calculate the intersection matrix of negative curves on $Y=Y_{d}$.
In particular, we can calculate the Dynkin type of $Y$ and the number of NECs on $Y$.

Indeed, it is determined by that of $Y_0$ and the choice of $\{f_i\}_{1 \leq i \leq d}$.
When $Y_0$ is an elliptic surface, comparing Theorem \ref{extell} and \cite[Theorem 4.1]{MP86}, there is an elliptic surface $Y_0'$ in characteristic zero with the same root lattice associated with the reduced fibers.
Since each negative curve is a section or is contained in a reducible fiber of the elliptic fibration, the intersection matrix of negative curves of $Y_0$ is the same as that of $Y_0'$ by Lemma \ref{lem:ellconfig}.
On the other hand, when $Y_0$ is a quasi-elliptic surface, we obtain the intersection matrix of negative curves on $Y_0$ by \cite[Figures 1--4]{KN}.
\end{rem}

In the remaining of this section, we follow the notation in Tables \ref{extell2}--\ref{q-ell3}.
In addition, we use the following notation.

\begin{defn}
We denote by $Y^e_d(*)$ the successive blow-down of $(-1)$-curves as in Lemma \ref{reduction} from an extremal rational elliptic surface or a rational quasi-elliptic surface of type $e$ such that the anti-canonical degree is $d$ and the configuration of $(-2)$-curves is the Dynkin diagram $*$.
\end{defn}

\subsection{Characteristic two}

In this subsection, we treat the case where $p=2$.
Let $X$ be a singular Du Val del Pezzo surface of Picard rank one.
By the same proof as in \cite[pp.14--15]{Ye}, we obtain that $K_X^2 \neq 7, 9$, and $X$ is the quadric cone in $\PP^3_k$ when $K_X^2=8$.
For this reason, we assume that $K_X^2 \leq 6$.
We follow the notation of Lemma \ref{reduction}.
By Remark \ref{rem:matrix}, we already know the intersection matrix of negative curves on $Y_i$ for $i\geq 0$.

We start with the case where $K_X^2=1$.
The pairs of $Y_0$ and $Y=Y_1$ are listed as in Table \ref{table:deg1},
where $n$ is the number of the NECs on $Y_1$.
    \begin{table}[ht]
\caption{Isomorphism classes of the pairs $(Y_0, Y_1)$ in $p=2$}
  \begin{tabular}{|c|c|c||c|c|c|} \hline
       Type of $Y_0$       &$Y_1$        &$n$& Type of $Y_0$      &$Y_1$        &$n$\\ \hline \hline
       I    &$Y^{\textup{I}}_1(D_8)$              &$2$& SII     &$Y^{\textup{SII}}_1(2A_4)$     &$0$\\ \hline
       II   &$Y^{\textup{II}}_1(E_8)$           &$1$& SIII    &$Y^{\textup{SIII}}_1(4A_2)$    &$0$\\ \hline
       III  &$Y^{\textup{III}}_1(A_1+A_7)$      &$1$& (a)     &$Y^{\textup{(a)}}_1(E_8)$      &$1$\\ \hline
       IV   &$Y^{\textup{IV}}_1(A_3+D_5)$       &$1$& (b)     &$Y^{\textup{(b)}}_1(D_8)$      &$2$\\ \hline
       V    &$Y^{\textup{V}}_1(A_1+E_7)$        &$1$& (c)     &$Y^{\textup{(c)}}_1(A_1+E_7)$  &$1$\\ \hline
       VI   &$Y^{\textup{VI}}_1(E_8)$           &$1$& (d)     &$Y^{\textup{(d)}}_1(2D_4)$     &$2$\\ \hline
       VII  &$Y^{\textup{VII}}_1(A_2+E_6)$      &$1$& (e)     &$Y^{\textup{(e)}}_1(2A_1+D_6)$ &$1$\\ \hline
       VIII &$Y^{\textup{VIII}}_1(A_1+A_2+A_5)$ &$0$& (f)     &$Y^{\textup{(f)}}_1(4A_1+D_4)$ &$1$\\ \hline
       IX   &$Y^{\textup{IX}}_1(A_2+E_6)$       &$1$& (g)     &$Y^{\textup{(g)}}_1(8A_1)$     &$0$\\ \hline
       SI   &$Y^{\textup{SI}}_1(A_8)$           &$2$&           &                      & \\ \hline
  \end{tabular}
  \label{table:deg1}
\end{table}

The isomorphism class of $Y_1$ is independent of the choice of $E_{f_1}$ by virtue of the $\MW(Y_0)$-action.
On the other hand, $Y_0$ is the blow-up of $Y_1$ at the base point of $|-K_{Y_1}|$.
Hence there is one-to-one correspondence between the isomorphism classes of $Y_0$ and those of $Y_1$.
Theorems \ref{extell} and \ref{q-ell} now show the assertion (1) of Theorem \ref{isom} in the case where $K_X^2=1$ and $p=2$.

Next, we consider the case where $K_X^2=2$.
Then $Y=Y_2$, which is the blow-down of an NEC in one of $Y_1$ listed in Table \ref{table:deg1}.
The pairs of $Y_1$ and $Y_2$ are listed as in Table \ref{table:deg2}, where $n$ is the number of the NECs on $Y_2$. 

    \begin{table}[ht]
\caption{Isomorphism classes of the pairs $(Y_1, Y_2)$ in $p=2$}
  \begin{tabular}{|c|c|c||c|c|c|} \hline
       $Y_1$                &$Y_2$                  &$n$& $Y_1$      &$Y_2$                             &$n$\\ \hline \hline
       $Y^{\textup{I}}_1(D_8)$         & $Y^{\textup{I}}_2(A_7)$          &$2$& $Y^{\textup{SI}}_1(A_8)$       & $Y^{\textup{SI}}_2(A_2+A_5)$   &$1$\\ \hline
       $Y^{\textup{I}}_1(D_8)$         & $Y^{\textup{I}}_2(A_1+D_6)$      &$1$& $Y^{\textup{(a)}}_1(E_8)$      & $Y^{\textup{(a)}}_2(E_7)$      &$1$\\ \hline
       $Y^{\textup{II}}_1(E_8)$      & $Y^{\textup{II}}_2(E_7)$       &$1$& $Y^{\textup{(b)}}_1(D_8)$      & $Y^{\textup{(b)}}_2(A_7)$      &$2$\\ \hline
       $Y^{\textup{III}}_1(A_1+A_7)$ & $Y^{\textup{III}}_2(A_1+2A_3)$ &$0$& $Y^{\textup{(b)}}_1(D_8)$      & $Y^{\textup{(b)}}_2(A_1+D_6)$  &$1$\\ \hline
       $Y^{\textup{IV}}_1(A_3+D_5)$  & $Y^{\textup{IV}}_2(A_1+2A_3)$  &$0$& $Y^{\textup{(c)}}_1(A_1+E_7)$  & $Y^{\textup{(c)}}_2(A_1+D_6)$  &$1$\\ \hline
       $Y^{\textup{V}}_1(A_1+E_7)$   & $Y^{\textup{V}}_2(A_1+D_6)$    &$1$& $Y^{\textup{(d)}}_1(2D_4)$     & $Y^{\textup{(d)}}_2(3A_1+D_4)$ &$0$\\ \hline
       $Y^{\textup{VI}}_1(E_8)$      & $Y^{\textup{VI}}_2(E_7)$       &$1$& $Y^{\textup{(e)}}_1(2A_1+D_6)$ & $Y^{\textup{(e)}}_2(3A_1+D_4)$ &$0$\\ \hline
       $Y^{\textup{VII}}_1(A_2+E_6)$ & $Y^{\textup{VII}}_2(A_2+A_5)$  &$1$& $Y^{\textup{(f)}}_1(4A_1+D_4)$ & $Y^{\textup{(f)}}_2(7A_1)$     &$0$\\ \hline
       $Y^{\textup{IX}}_1(A_2+E_6)$  & $Y^{\textup{IX}}_2(A_2+A_5)$   &$1$&                       & &\\ \hline
  \end{tabular}
  \label{table:deg2}
\end{table}

\begin{rem}\label{deg2rem}
The isomorphism class of $Y^{\textup{SI}}_2(A_2+A_5)$ is independent of the choice of $E_{f_2}$ because $\MW(Y_0^{\textup{SI}}) \cong \ZZ/3\ZZ$ and two NECs in $Y^{\textup{SI}}_1(A_8)$ are the images of the sections in $Y_0^{\textup{SI}}$.
The isomorphism class of $Y^{\textup{(d)}}_2(3A_1+D_4)$ is also
independent of the choice of $E_{f_2}$ 
because it maps to the other NEC in $Y^{\textup{(d)}}_1(2D_4)$ by the involution $\tau$ on $Y^{\textup{(d)}}_1(2D_4)$ constructed as in the proof of \cite[Proposition 5.17]{KN}.
\end{rem}

\begin{lem}\label{deg2isom}
We have the following isomorphisms:
\begin{enumerate}
    \item[\textup{(1)}] $Y^{\textup{I}}_2(A_7) \cong Y^{\textup{(b)}}_2(A_7)$.
    \item[\textup{(2)}] $Y^{\textup{I}}_2(A_1+D_6) \cong Y^{\textup{V}}_2(A_1+D_6)$.
    \item[\textup{(3)}] $Y^{\textup{III}}_2(A_1+2A_3) \cong Y^{\textup{IV}}_2(A_1+2A_3)$.
    \item[\textup{(4)}] $Y^{\textup{VII}}_2(A_2+A_5) \cong Y^{\textup{IX}}_2(A_2+A_5) \cong Y^{\textup{SI}}_2(A_2+A_5)$.
    \item[\textup{(5)}] $Y^{\textup{(b)}}_2(A_1+D_6) \cong Y^{\textup{(c)}}_2(A_1+D_6)$.
    \item[\textup{(6)}] $Y^{\textup{(d)}}_2(3A_1+D_4) \cong Y^{\textup{(e)}}_2(3A_1+D_4)$.
\end{enumerate}
\end{lem}

\begin{proof}
We give the proof only for the assertions (1), (2), (5), and (6): the proof of the assertions (3) and (4) run as in \cite[Claim 4.5 (4) and (2)]{Ye} respectively.

\noindent (1): Let $Z_{\alpha}$ be the extremal rational elliptic surface of type I whose j-invariant is $\alpha$.
Since each $Y^{\textup{I}}_2(A_7)$ is constructed from $Z_{\alpha}$ by blowing down all the sections, there is one to one correspondence between the isomorphism classes of $Y^{\textup{I}}_2(A_7)$ and those of $Z_{\alpha}$.
Note that $\alpha \in k^*$ by \cite[p.432]{Lang2}.
In particular, $Z_{\alpha}$ is a unique rational elliptic surface containing an $\textup{I}^*_4$-fiber and a smooth fiber $j$-invariant $\alpha$.

On the other hand, there is a unique rational quasi-elliptic surface of type (b).
Since it has exactly two sections, there is a unique $Y'=Y^{\textup{(b)}}_2(A_7)$.
Hence it suffices to show that for any $\alpha \in k^*$, we can construct $Z_\alpha$ by blowing up at some two points in $Y^{\textup{(b)}}_2(A_7)$.

Fix coordinates $[x: y: z]$ of $\PP^2_k$ and take $C \coloneqq \{x^3+y^2z=0\}$.
By \cite[Remark 4]{Ito2}, $Y'$ is also obtained from $\PP^2_k$ by blowing up four points on $C$ infinitely near $[0:1:0]$ and three points on $C$ infinitely near $[1:1:1]$.
Moreover, the push forward of $|-K_{Y'}|$ to $\PP^2_k$ is generated by $x^3+y^2z, (x+z)^2z$ and $(x+z)(y+z)z$.
Now take $C_\alpha$ as the strict transform of $\{x^3+y^2z+\alpha^{\frac18}(x+z)(y+z)z=0\} \subset \PP^2_k$ in $Y'$.
Then $C_\alpha$ is a smooth member of $|-K_{Y'}|$ whose $j$-invariant equals $\alpha$.
Let $Z'$ be the blow-up of $Y'$ at the intersection of $C_\alpha$ and two NECs on $Y'$.
Then it contains a $\textup{I}^*_4$-fiber and the strict transform of $C_\alpha$ as a fiber.
Therefore $Z' \cong Z_{\alpha}$ and $Y^{\textup{I}}_2(A_7) \cong Y^{\textup{(b)}}_2(A_7)$.

\noindent (2): Let $Z \to Y^{\textup{I}}_2(A_1+D_6)$ be the blow-up at a general point of the unique NEC.
Then $|-K_Z|$ has a smooth member and $Z$ contains eight $(-2)$-curves whose configuration is the Dynkin diagram $A_1+E_7$.
Since anti-canonical members of $Y^{\textup{(c)}}_1(A_1+E_7)$ are all singular, Table \ref{table:deg1} now shows that $Z \cong Y^{\textup{V}}_1(A_1+E_7)$.
Hence $Y^{\textup{I}}_2(A_1+D_6) \cong Y^{\textup{V}}_2(A_1+D_6)$.

\noindent (5): It follows from \cite[Proposition 5.4 and Remark 5.5]{KN}.

\noindent (6): It follows from \cite[Proposition 5.7 and Remark 5.8]{KN}.
\end{proof}

\begin{lem}\label{deg2isom-2}
We have the following:
\begin{enumerate}
    \item[\textup{(1)}] $Y^{\textup{I}}_2(A_1+D_6) \not \cong Y^{\textup{(b)}}_2(A_1+D_6)$.
    \item[\textup{(2)}] $Y^{\textup{(a)}}_2(E_7) \not \cong Y^{\textup{II}}_2(E_7) \text{ and } Y^{\textup{(a)}}_2(E_7) \not \cong Y^{\textup{VI}}_2(E_7)$.
\end{enumerate}
\end{lem}

\begin{proof}
A general anti-canonical member of $Y^{\textup{I}}_2(A_1+D_6)$ is smooth and that of $Y^{\textup{(b)}}_2(A_1+D_6)$ is singular by \cite[Remark 5.5]{KN}. 
Hence we have the assertion (1).
Similarly, \cite[Remark 5.2]{KN} gives the assertion (2).
\end{proof}

We will show that $Y^{\textup{II}}_2(E_7) \not \cong Y^{\textup{VI}}_2(E_7)$ in Lemma \ref{lem:E_7}.
In conclusion, there are $10$ isomorphism classes of $X$ with $K_X^2=2$.

Next, we deal with the case where $K_X^2=3$.
Then $Y=Y_3$, which is the blow-down of an NEC in one of $Y_2$ listed in Table \ref{table:deg2}.
The pairs of $Y_2$ and $Y_3$ are listed as in Table \ref{table:deg3}, where $n$ is the number of the NECs on $Y_3$.

    \begin{table}[ht]
\caption{Isomorphism classes of the pairs $(Y_2, Y_3)$ in $p=2$}
  \begin{tabular}{|c|c|c||c|c|c|} \hline
       $Y_2$                &$Y_3$              &$n$& $Y_2$                 &$Y_3$                 &$n$\\ \hline \hline
       $Y^{\textup{I}}_2(A_7)$         & $Y^{\textup{I}}_3(A_1+A_5)$  &$1$& $Y^{\textup{VII}}_2(A_2+A_5)$  & $Y^{\textup{VII}}_3(3A_2)$    &$0$\\ \hline
       $Y^{\textup{II}}_2(E_7)$      & $Y^{\textup{II}}_3(E_6)$   &$1$& $Y^{\textup{(a)}}_2(E_7)$      & $Y^{\textup{(a)}}_3(E_6)$     &$1$\\ \hline
       $Y^{\textup{V}}_2(A_1+D_6)$     & $Y^{\textup{V}}_3(A_1+A_5)$  &$1$& $Y^{\textup{(b)}}_2(A_1+D_6)$  & $Y^{\textup{(b)}}_3(A_1+A_5)$ &$1$\\ \hline
       $Y^{\textup{VI}}_2(E_7)$      & $Y^{\textup{VI}}_3(E_6)$   &$1$&       &   &\\ \hline
  \end{tabular}
  \label{table:deg3}
\end{table}

\begin{rem}\label{deg3rem}
The isomorphism class of $Y^{\textup{I}}_3(A_1+A_5)$ is independent of the choice of $E_{f_3}$ because the $\MW(Y_0)$-action naturally descends to $Y^{\textup{I}}_2(A_7)$, which sends one NEC to the other.
\end{rem}

\begin{lem}\label{deg3isom}
We have the isomorphisms \[Y^{\textup{V}}_3(A_1+A_5) \cong Y^{\textup{I}}_3(A_1+A_5) \cong Y^{\textup{(b)}}_3(A_1+A_5).\]
\end{lem}

\begin{proof}
It follows from Lemma \ref{deg2isom} (1) and (2).
\end{proof}

\begin{lem}\label{E_6}
It holds that $Y^{\textup{(a)}}_3(E_6) \cong Y^{\textup{II}}_3(E_6)$ and $Y^{\textup{(a)}}_3(E_6) \not \cong Y^{\textup{VI}}_3(E_6)$.
\end{lem}

\begin{proof}
Fix coordinates $[x: y: z]$ of $\PP^2_k$ and let $C \coloneqq \{x^3+y^2z=0\}$.
By \cite[Remark 5.2]{KN}, $Y^{\textup{(a)}}_3(E_6)$ is obtained by blowing up six points on $C$ infinitely near $t \coloneqq [0:1:0]$.
The anti-canonical linear system of $Y^{\textup{(a)}}_3(E_6)$ corresponds to the linear system $\Lambda$ of cubic curves intersecting with $C$ at $t$ with multiplicity at least six.
Then $\Lambda=\{ a(x^3+y^2z)+bz^3+cxz^2+dyz^2|[a:b:c:d] \in \PP^3_k \}$.
It is easy to check that $[0:1:0:0]$ corresponds to the member $3\{z=0\}$ and the locus of singular members of $\Lambda$ is $\{ad=0\}$.
In particular, a pencil in $\Lambda$ passing through $[0:1:0:0]$ either consists of singular members or contains exactly one singular member, which is $3\{z=0\}$.

On the other hand, we recall that the configuration of singular fibers of the extremal rational elliptic surface of type II (resp.\,VI) is $(\textup{II}^*)$ (resp.\,$(\textup{II}^*, \textup{I}_1)$), where we use Kodaira's notation.
Since an elimination of a rational map associated to a general pencil in $\Lambda$ passing through $[0:1:0:0]$ gives the extremal rational elliptic surface of type II, we conclude that $Y^{\textup{(a)}}_3(E_6) \cong Y^{\textup{II}}_3(E_6)$.
On the other hand, there is no pencil in $\Lambda$ passing through $[0:1:0:0]$ which contains exactly two singular members. Hence $Y^{\textup{(a)}}_3(E_6) \not \cong Y^{\textup{VI}}_3(E_6)$.
\end{proof}

\begin{lem}\label{lem:E_7}
It holds that $Y^{\textup{II}}_2(E_7) \not \cong Y^{\textup{VI}}_2(E_7)$
\end{lem}

\begin{proof}
If $Y^{\textup{II}}_2(E_7) \cong Y^{\textup{VI}}_2(E_7)$, then
we obtain $Y^{\textup{II}}_3(E_6) \cong Y^{\textup{VI}}_3(E_6)$, a contradiction with Lemma \ref{E_6}.
\end{proof}

Therefore there are four isomorphism classes of $X$ with $K_X^2=3$.

Next, we deal with the case where $K_X^2=4$.
Then $Y=Y_4$, which is the blow-down of an NEC in one of $Y_3$ listed in Table \ref{table:deg3}.
The pairs of $Y_3$ and $Y_4$ are listed as in Table \ref{table:deg4}, where $n$ is the number of the NECs on $Y_4$.

    \begin{table}[ht]
\caption{Isomorphism classes of the pairs $(Y_3, Y_4)$ in $p=2$}
  \begin{tabular}{|c|c|c|} \hline
       $Y_3$                &$Y_4$                   &$n$\\ \hline \hline
       $Y^{\textup{I}}_3(A_1+A_5)$     & $Y^{\textup{I}}_4(2A_1+A_3)$      &$0$ \\ \hline
       $Y^{\textup{VI}}_3(E_6)$      & $Y^{\textup{VI}}_4(D_5)$        &$1$ \\ \hline
       $Y^{\textup{(a)}}_3(E_6)$  & $Y^{\textup{(a)}}_4(D_5)$          &$1$\\ \hline
  \end{tabular}
  \label{table:deg4}
\end{table}

\begin{lem}\label{D_5}
It holds that $Y^{\textup{(a)}}_4(D_5) \not \cong Y^{\textup{VI}}_4(D_5)$.
\end{lem}

\begin{proof}
We follow the notation of the proof of Lemma \ref{E_6}.
Then $Y^{\textup{(a)}}_4(D_5)$ is obtained by blowing up five points on $C \subset \PP^2_k$ infinitely near $t$ and the anti-canonical linear system of $Y^{\textup{(a)}}_4(D_5)$ corresponds to the linear system $\Lambda'=\{ a(x^3+y^2z)+bz^3+cxz^2+dyz^2+ex^2z|[a:b:c:d:e] \in \PP^4_k \}$.
It is easy to check that $[0:1:0:0:0]$ corresponds to the member $3\{z=0\}$ and the locus of singular members of $\Lambda'$ is $\{ad=0\}$.
In particular, there is no pencil in $\Lambda'$ passing through $[0:1:0:0:0]$ which contains exactly two singular members.
Hence $Y^{\textup{(a)}}_4(D_5) \not \cong Y^{\textup{VI}}_4(D_5)$.
\end{proof}

Therefore there are three isomorphism classes of $X$ with $K_X^2=4$.

Next, we deal with the case where $K_X^2=5$.
Then $Y=Y_5$, which is the blow-down of an NEC in one of $Y_4$ listed in Table \ref{table:deg4}.
The pairs of $Y_4$ and $Y_5$ are listed as in Table \ref{table:deg4}, where $n$ is the number of the NECs on $Y_5$.

    \begin{table}[ht]
\caption{Isomorphism classes of the pairs $(Y_4, Y_5)$ in $p=2$}
  \begin{tabular}{|c|c|c|} \hline
       $Y_4$                &$Y_5$                   &$n$\\ \hline \hline
       $Y^{\textup{VI}}_4(D_5)$      & $Y^{\textup{VI}}_5(A_4)$        &$1$ \\ \hline
       $Y^{\textup{(a)}}_4(D_5)$  & $Y^{\textup{(a)}}_5(A_4)$          &$1$\\ \hline
  \end{tabular}
  \label{table:deg5}
\end{table}


\begin{lem}\label{A_4}
It holds that $Y^{\textup{(a)}}_5(A_4) \cong Y^{\textup{VI}}_5(A_4)$.
\end{lem}

\begin{proof}
We follow the notation of the proof of Lemma \ref{E_6}.
Then $Y^{\textup{(a)}}_5(A_4)$ is obtained by blowing up four points on $C \subset \PP^2_k$ infinitely near $t$ and the anti-canonical linear system of $Y^{\textup{(a)}}_5(A_4)$ corresponds to the linear system $\Lambda''=\{ a(x^3+y^2z)+bz^3+cxz^2+dyz^2+ex^2z+fxyz|[a:b:c:d:e:f] \in \PP^5_k \}$.
Since $\{x^3+y^2z+xyz=0\}$ is a nodal cubic, an elimination of a rational map associated to the pencil $\langle z^3, x^3+y^2z+xyz\rangle \subset \Lambda''$ gives the extremal rational elliptic surface of type VI.
Hence $Y^{\textup{(a)}}_5(A_4) \cong Y^{\textup{VI}}_5(A_4)$.
\end{proof}

Finally, we deal with the case where $K_X^2=5$.
Then $Y=Y_6$, which is the blow-down of an NEC in one of $Y_5$ listed in Table \ref{table:deg5}.
The pairs of $Y_5$ and $Y_6$ are listed as in Table \ref{table:deg4}, where $n$ is the number of the NECs on $Y_6$.

    \begin{table}[ht]
\caption{Isomorphism classes of the pairs $(Y_5, Y_6)$ in $p=2$}
  \begin{tabular}{|c|c|c|} \hline
       $Y_5$                &$Y_6$                   &$n$\\ \hline \hline
       $Y^{\textup{(a)}}_5(A_4)$  & $Y^{\textup{(a)}}_6(A_1+A_2)$          &$0$\\ \hline
  \end{tabular}
  \label{table:deg6}
\end{table}


Now, we calculate the defining equations of some Du Val del Pezzo surfaces, some of which we will use in the proof of Theorems \ref{isom} and \ref{F split Intro}. 


\begin{prop}\label{defeq7A_1}
A Du Val del Pezzo surface $X(7A_1)$ is determined unique up to isomorphisms, and there are coordinates $[x: y: z: w]$ of $\PP_k(1,1,1,2)$ such that the defining equation of the Du Val del Pezzo surface $X(7A_1)$ is 
\begin{align}\label{eq:7A1}
w^2+xyz(x+y+z)=0.    
\end{align}
\end{prop}

\begin{proof}
Take $Y$ as the minimal resolution of $X$.
Let $[s:t:u]$ be coordinates of $\PP^2_k$.
Then there is the blow-down $h \colon Y \to \PP^2_k$ such that $h(E_h) \subset \PP^2_k$ is the set of closed points defined over $\FF_2$ by \cite[Proposition 5.9]{KN}. 
In particular, $X(7A_1)$ is unique up to isomorphisms.
Set 
\begin{align}\label{eq:7A_1coord}
\begin{split}
&x \coloneqq st(s+t), y \coloneqq tu(t+u), z \coloneqq us(u+s), \text{ and }\\
&w \coloneqq stu(s+t)(t+u)(u+s).
\end{split}
\end{align}
Then $x,y,z \in H^0(\PP^2_k, h_*\sO_Y(-K_Y))$ and $w \in H^0(\PP^2_k, h_*\sO_Y(-2K_Y))$
because $H^0(\PP^2_k, h_*\sO_Y(-nK_Y)) \subset H^0(\PP^2_k, \sO_{\PP^2_k}(3n))$ consists of elements which have zero of order at least $n$ at each points in $h(E_h)$ for $n \geq 1$.
Moreover, it is easy to check that $\{x^2, y^2, z^2, xy, yz, zx, w\}$ is a basis of $H^0(\PP^2_k, h_*\sO_Y(-2K_Y))$.
Hence $X$ is the closure of the image of the rational map
\begin{align*}
    \Phi \colon \PP^2_k &\dashrightarrow \PP_k(1,1,1,2)\\ 
    [s:t:u] & \mapsto [x:y:z:w].
\end{align*}
Substituting (\ref{eq:7A_1coord}) into the left-hand side of (\ref{eq:7A1}) yields zero as a polynomial of $s,t$, and $u$.
Therefore, (\ref{eq:7A1}) is a defining equation of $X(7A_1)$. 
\end{proof}


\begin{prop}\label{defeqNBdeg1char2}
Let $X$ be a Du Val del Pezzo surface of degree one of Dynkin type $4A_1+D_4$ or $8A_1$.
Set $\mathcal{D}_n \subset \PP^n_k$ as the complement of the union of all the hyperplane sections defined over $\FF_2$ for $n=1, 2$. 
Then we can choose coordinates of $\PP_k(1,1,2,3)$ such that the defining equation of $X$ in $\PP_k(1,1,2,3)$ is 
\begin{align}
w^2+z^3+abx^2z^2+y^4z+(a^2+ab+b^2)x^2y^2z+ab(a+b)x^3yz=0 
\end{align} 
for some $[a:b] \in \mathcal{D}_1$ when $\Dyn(X)=4A_1+D_4$, and 
\begin{small}
\begin{align}\label{eq:8A1}
\begin{split}
&w^2+abcz^3+((ab+bc+ca)^2+abc(a+b+c))y^2z^2\\
&+(a+b+c)(a+b)(b+c)(c+a)xyz^2\\
&+(ab+bc+ca)^2x^2z^2
+(a+b+c)^2(a+b)(b+c)(c+a)xy^3z\\
&+(a+b+c)^2((a+b+c)^3+abc)x^2y^2z
+(a+b+c)^2(a+b)(b+c)(c+a)x^3yz\\
&+(a+b+c)^2abcx^4z
+(a+b)^2(b+c)^2(c+a)^2y^6
+((a+b+c)^3+abc)^2x^2y^4\\
&+(a+b)^2(b+c)^2(c+a)^2x^4y^2  
+a^2b^2c^2x^6=0
\end{split}
\end{align}
\end{small}
for some $[a:b:c] \in \mathcal{D}_2$ when $\Dyn(X)=8A_1$. 
\end{prop}

\begin{proof}
We give the proof only for the case $\Dyn(X) = 8A_1$; the other cases are left to the reader.
Let $Y$ be the minimal resolution of $X$.
Fix coordinates $[s: t: u]$ of $\PP^2_k$.
Then there is $t=[a:b:c] \in \mathcal{D}_2$ and the blow-down $h \colon Y \to \PP^2_k$ such that $t \in h(E_h)$ and $h(E_h) \setminus \{t\} \subset \PP^2_k$ is the set of closed points defined over $\FF_2$ by \cite[Proposition 5.21]{KN}.
Set 
\begin{align}\label{eq:8A_1coord}
\begin{split}
x \coloneqq& c(a+b)st(s+t)+a(b+c)tu(t+u)+b(c+a)us(u+s),\\
y \coloneqq& c^2st(s+t)+a^2tu(t+u)+b^2us(u+s), \\
z \coloneqq& ((b+c)s+(c+a)t+(a+b)u)^2 stu(s+t+u),\text{ and }\\
w \coloneqq& ((b+c)s+a(t+u))((c+a)t+b(u+s))((a+b)u+c(s+t))\\
&\times stu(s+t)(t+u)(u+s).
\end{split}
\end{align}
Then $x, y \in H^0(\PP^2_k, h_*\sO_Y(-K_Y))$, $z \in H^0(\PP^2_k, h_*\sO_Y(-2K_Y))$ and $w \in H^0(\PP^2_k, h_*\sO_Y(-3K_Y))$ 
because $H^0(\PP^2_k, h_*\sO_Y(-nK_Y)) \subset H^0(\PP^2_k, \sO_{\PP^2_k}(3n))$ consists of functions which has zero of order at least $n$ at each points in $h(E_h)$ for $n \geq 1$.
Moreover, it is easy to check that $\{x^3, x^2y, xy^2, y^3, xz, yz, w\}$ is a basis of $H^0(\PP^2_k, h_*\sO_Y(-3K_Y))$.
Hence $X$ is the closure of the image of the rational map
\begin{align*}
    \Phi \colon \PP^2_k &\dashrightarrow \PP_k(1,1,2,3)\\ 
    [s:t:u] & \mapsto [x:y:z:w].
\end{align*}
Substituting (\ref{eq:8A_1coord}) into the left-hand side of (\ref{eq:8A1}) yields zero as a polynomial of $s,t$, and $u$.
Therefore, (\ref{eq:8A1}) is a defining equation of $X(8A_1)$. 
\end{proof}

\begin{prop}\label{Artin2}
Let $X$ be a Du Val del Pezzo surface of $\rho(X)=1$ and $p=2$.
If $\Dyn (X) \neq D_8, 2D_4, 4A_1+D_4$, or $8A_1$, then its isomorphism class is uniquely determined by its Dynkin type with Artin coindices.
\end{prop}

\begin{proof}
Let $Y$ be the minimal resolution of $X$.
By Theorem \ref{isom} (1), we may assume that $\Dyn(X)=D_5, E_6, E_7, A_1+D_6, E_8, A_1+E_7$, or $A_2+E_6$.

When $\Dyn(X) \neq A_1+E_7$ or $A_2+E_6$ in addition, we calculate the defining equation of $X$ as in the proof of Propositions \ref{defeq7A_1} and \ref{defeqNBdeg1char2};
firstly, we choose a suitable blow-down $h \colon Y \to \PP^2_k$ and calculate a basis of $\Lambda \coloneqq h_*|-nK_Y|$ with $n=1$ (resp.~$=2$, $=3$) when $K_Y^2 \geq 3$ (resp.~$=2$, $=1$).
Then $X$ is the closure of the image of the map from $\PP^2_k$ to $\PP \coloneqq \PP^{d}_k$ with $d=K_Y^2$ (resp.~$\PP_k(1,1,1,2)$, $\PP_k(1,1,2,3)$) defined by $\Lambda$ when $K_Y^2 \geq 3$ (resp.~$=2$, $=1$).
Finally we compute the defining equation of $X$ in $\PP$ to determine $\Dyn'(X)$.
As a result, we get the defining equation of $X$ as in Table \ref{table:Artin2}, where $[x_0:\cdots:x_4]$ stands for coordinates of $\PP^4_k$ and $[x:y:z:w]$ stands for coordinates of $\PP^3_k, \PP_k(1,1,1,2)$, or $\PP_k(1,1,2,3)$.


    \begin{table}[ht]
\caption{Artin coindices and defining equations of Du Val del Pezzo surfaces of Dynkin type $D_5$, $E_6$, $E_7$, $A_1+D_6$, or $E_8$ in $p=2$}
  \begin{tabular}{|c|c|l|c|} \hline
         $K_Y^2$                   &  $Y$ &\multicolumn{1}{l|}{defining equation of $X \subset \PP$}                 & $\Dyn'(X)$\\ \hline \hline
    \multirow{2}{*}{$4$}                        &  $Y^{\textup{(a)}}_4(D_5)$ &$x_2^2+x_1x_4, x_0x_1+x_2x_4+x_3^2$       & $D_5^0$\\ \cline{2-4}
                            &  $Y^{\textup{VI}}_4(D_5)$ &$x_2^2+x_1x_4, x_0x_1+x_2x_4+x_3^2+x_2x_3$& $D_5^1$\\ \hline \hline
    \multirow{2}{*}{$3$}                         &  $Y^{\textup{(a)}}_3(E_6)$& $wz^2+x^3+y^2z       $& $E_6^0$\\ \cline{2-4}
                            & $Y^{\textup{VI}}_3(E_6)$ & $wz^2+x^3+y^2z+xyz   $& $E_6^1$\\ \hline \hline
    \multirow{5}{*}{$2$}                         & $Y^{\textup{(a)}}_2(E_7)$   & $w^2+yz^3+xy^3       $& $E_7^0$\\ \cline{2-4}
                            & $Y^{\textup{II}}_2(E_7)$    & $w^2+yz^3+xy^3+y^2w   $& $E_7^2$\\ \cline{2-4}
                            & $Y^{\textup{VI}}_2(E_7)$   & $w^2+yz^3+xy^3+yzw   $& $E_7^3$\\ \cline{2-4}
                            & $Y^{\textup{(b)}}_2(A_1+D_6)$  & $w^2+xyz^2+y^3z$& $A_1+D_6^0$\\ \cline{2-4}
                            & $Y^{\textup{I}}_2(A_1+D_6)$    & $w^2+xyz^2+y^3z+yzw$ & $A_1+D_6^2$\\ \hline \hline
    \multirow{3}{*}{$1$}                         & $Y^{\textup{(a)}}_1(E_8)$                 & $w^2+z^3+xy^5       $& $E_8^0$\\ \cline{2-4}
                            & $Y^{\textup{II}}_1(E_8)$                  & $w^2+z^3+xy^5+y^3w  $& $E_8^3$\\ \cline{2-4}
                            & $Y^{\textup{VI}}_1(E_8)$                  & $w^2+z^3+xy^5+yzw   $& $E_8^4$\\ \hline
  \end{tabular}
\label{table:Artin2}
\end{table}

Finally, suppose that $\Dyn(X)=A_1+E_7$ or $A_2+E_6$.
Then a suitable choice of blow-down $Y \to Y'$ gives the minimal resolution $Y'$ of a del Pezzo surface of Picard rank one of type $E_7$ or $E_6$ as in Table \ref{Artin2-1'}. 

    \begin{table}[ht]
\caption{Artin coindices of Du Val del Pezzo surfaces of Dynkin type $A_1+E_7$ or $A_2+E_6$ in $p=2$}
  \begin{tabular}{|c|c|c|c|} \hline
         $Y$ &$Y'$& $\Dyn'(X)$ \\ \hline \hline
         $Y^{\textup{(c)}}_1(A_1+E_7)$& $Y^{\textup{(a)}}_2(E_7)$& $A_1+E_7^0$\\ \hline
         $Y^{\textup{V}}_1(A_1+E_7)$  & $Y^{\textup{VI}}_2(E_7)$& $A_1+E_7^3$\\ \hline
         $Y^{\textup{VII}}_1(A_2+E_6)$& $Y^{\textup{(a)}}_3(E_6)$& $A_2+E_6^0$\\ \hline
         $Y^{\textup{XI}}_1(A_2+E_6)$ & $Y^{\textup{VI}}_3(E_6)$& $A_2+E_6^1$\\ \hline
  \end{tabular}
  \label{Artin2-1'}
\end{table}

Combining these results, we get the assertion.
\end{proof}


\subsection{Characteristic three}
In this subsection, we treat the case where $p=3$.
Let $X$ be a singular Du Val del Pezzo surface of Picard rank one.
As in the case where $p=2$, we may assume that $K_X^2 \leq 6$.
We follow the notation of Lemma \ref{reduction}.

We start with the case where $K_X^2=1$.
The pairs of $Y_0$ and $Y=Y_1$ are listed as in Table \ref{table:deg1'}, where $n$ is the number of the NECs on $Y_1$.
    \begin{table}[ht]
\caption{Isomorphism classes of the pairs $(Y_0, Y_1)$ in $p=3$}
  \begin{tabular}{|c|c|c||c|c|c|} \hline
       Type of $Y_0$       &$Y_1$       &$n$& Type of $Y_0$      &$Y_1$         &$n$\\ \hline \hline
       I    &$Y^{\textup{I}}_1(E_8)$           &$1$& X     &$Y^{\textup{X}}_1(D_8)$           &$2$\\ \hline
       II   &$Y^{\textup{II}}_1(A_8)$          &$2$& XI    &$Y^{\textup{XI}}_1(A_1+E_7)$      &$1$\\ \hline
       III  &$Y^{\textup{III}}_1(A_2+E_6)$     &$1$& SI    &$Y^{\textup{SI}}_1(A_1+A_7)$      &$1$\\ \hline
       IV   &$Y^{\textup{IV}}_1(E_8)$          &$1$& SII   &$Y^{\textup{SII}}_1(2A_4)$        &$0$\\ \hline
       V    &$Y^{\textup{V}}_1(A_1+E_7)$       &$1$& SIII  &$Y^{\textup{SIII}}_1(2A_1+2A_3)$  &$0$\\ \hline
       VI   &$Y^{\textup{VI}}_1(2D_4)$         &$2$& (1)   &$Y^{\textup{(1)}}_1(E_8)$         &$1$\\ \hline
       VII  &$Y^{\textup{VII}}_1(A_1+A_2+A_5)$ &$0$& (2)   &$Y^{\textup{(2)}}_1(A_2+E_6)$     &$1$\\ \hline
       VIII &$Y^{\textup{VIII}}_1(A_3+D_5)$    &$1$& (3)   &$Y^{\textup{(3)}}_1(4A_2)$        &$0$\\ \hline
       IX   &$Y^{\textup{IX}}_1(2A_1+D_6)$     &$1$& & &\\ \hline
  \end{tabular}
  \label{table:deg1'}
\end{table}

By virtue of the $\MW(Y_0)$-action, there is one to one correspondence between the isomorphism classes of $Y_0$ and those of $Y_1$.
Theorems \ref{extell} and \ref{q-ell} now show the assertion (1) of Theorem \ref{isom} in the case where $K_X^2=1$ and $p=3$.

Next, we consider the case where $K_X^2=2$.
Then $Y=Y_2$, which is the blow-down of an NEC in one of $Y_1$ listed in Table \ref{table:deg1'}.
The pairs of $Y_1$ and $Y_2$ are listed as in Table \ref{table:deg2'}, where $n$ is the number of the NECs on $Y_2$. 

    \begin{table}[ht]
\caption{Isomorphism classes of the pairs $(Y_1, Y_2)$ in $p=3$}
  \begin{tabular}{|c|c|c||c|c|c|} \hline
       $Y_1$                &$Y_2$                   &$n$& $Y_1$      &$Y_2$                           &$n$\\ \hline \hline
       $Y^{\textup{I}}_1(E_8)$         & $Y^{\textup{I}}_2(E_7)$           &$1$& $Y^{\textup{IX}}_1(2A_1+D_6)$ & $Y^{\textup{IX}}_2(3A_1+D_4)$ &$0$\\ \hline
       $Y^{\textup{II}}_1(A_8)$      & $Y^{\textup{II}}_2(A_2+A_5)$    &$1$& $Y^{\textup{X}}_1(D_8)$       & $Y^{\textup{X}}_2(A_7)$       &$2$\\ \hline
       $Y^{\textup{III}}_1(A_2+E_6)$ & $Y^{\textup{III}}_2(A_2+A_5)$   &$1$& $Y^{\textup{X}}_1(D_8)$       & $Y^{\textup{X}}_2(A_1+D_6)$   &$1$\\ \hline
       $Y^{\textup{IV}}_1(E_8)$      & $Y^{\textup{IV}}_2(E_7)$        &$1$& $Y^{\textup{XI}}_1(A_1+E_7)$  & $Y^{\textup{XI}}_2(A_1+D_6)$  &$1$\\ \hline
       $Y^{\textup{V}}_1(A_1+E_7)$   & $Y^{\textup{V}}_2(A_1+D_6)$     &$1$& $Y^{\textup{SI}}_1(A_1+A_7)$  & $Y^{\textup{SI}}_2(A_1+2A_3)$ &$0$\\ \hline
       $Y^{\textup{VI}}_1(2D_4)$     & $Y^{\textup{VI}}_2(3A_1+D_4)$   &$0$& $Y^{\textup{(1)}}_1(E_8)$     & $Y^{\textup{(1)}}_2(E_7)$     &$1$\\ \hline
       $Y^{\textup{VIII}}_1(A_3+D_5)$& $Y^{\textup{VIII}}_2(A_1+2A_3)$ &$0$& $Y^{\textup{(2)}}_1(A_2+E_6)$ & $Y^{\textup{(2)}}_2(A_2+A_5)$ &$1$\\ \hline
  \end{tabular}
  \label{table:deg2'}
\end{table}

Analysis similar to that in Remark \ref{deg2rem} shows that $Y^{\textup{II}}_2(A_2+A_5)$ and $Y^{\textup{VI}}_2(3A_1+D_4)$ are unique up to isomorphism.

\begin{lem}\label{deg2isom'}
We have the following isomorphisms:
\begin{enumerate}
    \item[\textup{(1)}] $Y^{\textup{III}}_2(A_2+A_5) \cong Y^{\textup{II}}_2(A_2+A_5) \cong Y^{\textup{(2)}}_2(A_2+A_5)$.
    \item[\textup{(2)}] $Y^{\textup{V}}_2(A_1+D_6) \cong Y^{\textup{X}}_2(A_1+D_6) \cong Y^{\textup{XI}}_2(A_1+D_6)$.
    \item[\textup{(3)}] $Y^{\textup{VI}}_2(3A_1+D_4) \cong Y^{\textup{IX}}_2(3A_1+D_4)$.
    \item[\textup{(4)}] $Y^{\textup{VIII}}_2(A_1+2A_3) \cong Y^{\textup{SI}}_2(A_1+2A_3)$.
\end{enumerate}
\end{lem}

\begin{proof}
We give the proof only for the isomorphism $Y^{\textup{II}}_2(A_2+A_5) \cong Y^{\textup{(2)}}_2(A_2+A_5)$:
the proof of the other assertions run as in \cite[Claim 4.5]{Ye}.

Let $Z \to Y^{\textup{(2)}}_2(A_2+A_5)$ be the blow-up at a general point of a $(-1)$-curve which is not an NEC.
Then $Z$ contains eight $(-2)$-curves whose configuration is the Dynkin diagram $A_8$.
Hence $Z=Y^{\textup{II}}_1(A_8)$ and $Y^{\textup{II}}_2(A_2+A_5) \cong Y^{\textup{(2)}}_2(A_2+A_5)$.
\end{proof}

\begin{lem}\label{E_7'}
It holds that $Y^{\textup{(1)}}_2(E_7) \cong Y^{\textup{I}}_2(E_7)$ and $Y^{\textup{(1)}}_2(E_7) \not \cong Y^{\textup{IV}}_2(E_7)$.
\end{lem}

\begin{proof}
First we recall the construction of $Y^{\textup{(1)}}_2(E_7)$.
Fix coordinates $[x: y: z]$ of $\PP^2_k$ and let $C \coloneqq \{x^3+y^2z=0\}$.
By \cite[Example 3.8]{Ito1}, an elimination of a rational map associated to the pencil $\langle x^3+y^2z, z^3\rangle$ is the rational quasi-elliptic surface of type (1).
Thus $Y^{\textup{(1)}}_2(E_7)$ is obtained by blowing up seven points on $C$ infinitely near $t \coloneqq [0:1:0]$ and the anti-canonical linear system of $Y^{\textup{(1)}}_2(E_7)$ corresponds to the linear system $\Lambda=\{ a(x^3+y^2z)+bz^3+cxz^2|[a:b:c] \in \PP^2_k \}$.
It is easy to check that $[0:1:0]$ corresponds to the member $3\{z=0\}$ and the locus of singular members of $\Lambda$ is $\{ac=0\}$.
In particular, a pencil in $\Lambda$ passing through $[0:1:0]$ either consists of singular members or contains exactly one singular member, which is $3\{z=0\}$.

On the other hand, we recall that the configuration of singular fibers of the extremal rational elliptic surface of type I (resp.\,IV) is $(\textup{II}^*)$ (resp.\,$(\textup{II}^*, \textup{I}_1)$), where we use Kodaira's notation.
Since an elimination of a rational map associated to a general pencil in $\Lambda$ passing through $[0:1:0]$ gives the extremal rational elliptic surface of type I, we conclude that $Y^{\textup{(1)}}_2(E_7) \cong Y^{\textup{I}}_2(E_7)$.
On the other hand, there is no pencil in $\Lambda$ passing through $[0:1:0]$ which contains exactly two singular members. Hence $Y^{\textup{(1)}}_2(E_7) \not \cong Y^{\textup{IV}}_2(E_7)$
\end{proof}

In conclusion, there are seven isomorphism classes of $X$ with $K_X^2=2$.

Next, we deal with the case where $K_X^2=3$.
Then $Y=Y_3$, which is the blow-down of an NEC in one of $Y_2$ listed in Table \ref{table:deg2'}.
The pairs of $Y_2$ and $Y_3$ are listed as in Table \ref{table:deg3'}, where $n$ is the number of the NECs on $Y_3$.

    \begin{table}[ht]
\caption{Isomorphism classes of the pairs $(Y_2, Y_3)$ in $p=3$}
  \begin{tabular}{|c|c|c||c|c|c|} \hline
       $Y_2$                &$Y_3$                  &$n$& $Y_2$                 &$Y_3$                 &$n$\\ \hline \hline
       $Y^{\textup{II}}_2(A_2+A_5)$  & $Y^{\textup{II}}_3(3A_2)$      &$0$& $Y^{\textup{X}}_2(A_7)$        & $Y^{\textup{X}}_3(A_1+A_5)$   &$1$\\ \hline
       $Y^{\textup{IV}}_2(E_7)$      & $Y^{\textup{IV}}_3(E_6)$       &$1$& $Y^{\textup{(1)}}_2(E_7)$      & $Y^{\textup{(1)}}_3(E_6)$     &$1$\\ \hline
       $Y^{\textup{V}}_2(A_1+D_6)$   & $Y^{\textup{V}}_3(A_1+A_5)$      &$1$&   &  &\\ \hline
  \end{tabular}
  \label{table:deg3'}
\end{table}

Analysis similar to that in Remark \ref{deg3rem} shows that $Y^{\textup{X}}_3(A_1+A_5)$ is unique up to isomorphism.

\begin{lem}
We have the isomorphism $Y^{\textup{V}}_3(A_1+A_5) \cong Y^{\textup{X}}_3(A_1+A_5)$.
\end{lem}

\begin{proof}
The assertion follows from Lemma \ref{deg2isom'} (2).
\end{proof}

\begin{lem}\label{E_6'}
It holds that $Y^{\textup{(1)}}_3(E_6) \not \cong Y^{\textup{IV}}_3(E_6)$.
\end{lem}

\begin{proof}
We follow the notation of the proof of Lemma \ref{E_7'}.
Then $Y^{\textup{(1)}}_3(E_6)$ is obtained by blowing up six points on $C \subset \PP^2_k$ infinitely near $t$ and the anti-canonical linear system of $Y^{\textup{(1)}}_3(E_6)$ corresponds to the linear system $\Lambda'=\{ a(x^3+y^2z)+bz^3+cxz^2+dyz^2|[a:b:c:d] \in \PP^3_k \}$.
It is easy to check that $[0:1:0:0]$ corresponds to the member $3\{z=0\}$ and the locus of singular members of $\Lambda'$ is $\{ac=0\}$.
In particular, there is no pencil in $\Lambda'$ passing through $[0:1:0:0]$ which contains exactly two singular members.
Hence $Y^{\textup{(1)}}_3(E_6) \not \cong Y^{\textup{IV}}_3(E_6)$.
\end{proof}

Therefore there are four isomorphism classes of $X$ with $K_X^2=3$.

Next, we deal with the case where $K_X^2=4$.
Then $Y=Y_4$, which is the blow-down of an NEC in one of $Y_3$ listed in Table \ref{table:deg3'}.
The pairs of $Y_3$ and $Y_4$ are listed as in Table \ref{table:deg4'}, where $n$ is the number of the NECs on $Y_4$.

    \begin{table}[ht]
\caption{Isomorphism classes of the pairs $(Y_3, Y_4)$ in $p=3$}
  \begin{tabular}{|c|c|c|} \hline
       $Y_3$                &$Y_4$                  &$n$\\ \hline \hline
       $Y^{\textup{IV}}_3(E_6)$      & $Y^{\textup{IV}}_4(D_5)$       &$1$ \\ \hline
       $Y^{\textup{V}}_3(A_1+A_5)$     & $Y^{\textup{V}}_4(2A_1+A_3)$     &$0$ \\ \hline
       $Y^{\textup{(1)}}_3(E_6)$     & $Y^{\textup{(1)}}_4(D_5)$      &$1$\\ \hline
  \end{tabular}
  \label{table:deg4'}
\end{table}

\begin{lem}\label{D_5'}
It holds that $Y^{\textup{(1)}}_4(D_5) \cong Y^{\textup{IV}}_4(D_5)$.
\end{lem}

\begin{proof}
We follow the notation of the proof of Lemma \ref{E_7'}.
Then $Y^{\textup{(1)}}_4(D_5)$ is obtained by blowing up five points on $C \subset \PP^2_k$ infinitely near $t$ and the anti-canonical linear system of $Y^{\textup{(1)}}_4(D_5)$ corresponds to the linear system $\Lambda''=\{ a(x^3+y^2z)+bz^3+cxz^2+dyz^2+ex^2z|[a:b:c:d:e] \in \PP^4_k \}$.
Since $\{x^3+y^2z+x^2z=0\}$ is a nodal cubic, an elimination of a rational map associated to the pencil $\langle z^3, x^3+y^2z+x^2z \rangle \subset \Lambda''$ gives the extremal rational elliptic surface of type IV.
Hence $Y^{\textup{(1)}}_4(D_5) \cong Y^{\textup{IV}}_4(D_5)$.
\end{proof}

Therefore there are two isomorphism classes of $X$ with $K_X^2=4$.

Finally, we deal with the case where $K_X^2 =5$ or $6$.
When $K_X^2 =5$ (resp. $K_X^2=6$), the pairs of $Y_4$ and $Y=Y_5$ (resp.\ $Y_5$ and $Y=Y_6$) are listed as in Table \ref{table:deg5'} (resp. Table \ref{table:deg6'}), where $n$ is the number of the NECs on $Y_5$ (resp.\ $Y_6$).

    \begin{table}[ht]
\caption{Isomorphism classes of the pairs $(Y_4, Y_5)$ in $p=3$}
  \begin{tabular}{|c|c|c|} \hline
       $Y_4$                &$Y_5$                  &$n$\\ \hline \hline
       $Y^{\textup{(1)}}_4(D_5)$     & $Y^{\textup{(1)}}_5(A_4)$      &$1$\\ \hline
  \end{tabular}
  \label{table:deg5'}
\end{table}

    \begin{table}[ht]
\caption{Isomorphism classes of the pairs $(Y_5, Y_6)$ in $p=3$}
  \begin{tabular}{|c|c|c|} \hline
       $Y_5$                &$Y_6$                  &$n$\\ \hline \hline
       $Y^{\textup{(1)}}_5(A_4)$     & $Y^{\textup{(1)}}_6(A_1+A_2)$      &$0$\\ \hline
  \end{tabular}
  \label{table:deg6'}
\end{table}


Now, we calculate the defining equations of some Du Val del Pezzo surfaces,  which we will use in the proof of Theorems \ref{isom} and \ref{F split Intro}. 


\begin{prop}\label{Artin3}
Let $X$ be a Du Val del Pezzo surface of $\rho(X)=1$ and $p=3$.
If $\Dyn (X) \neq 2D_4$, then its isomorphism class is uniquely determined by its Dynkin type with Artin coindices.
\end{prop}

\begin{proof}
We follow the notation of Proposition \ref{Artin2}.
By Theorem \ref{isom} (1), we may assume that $\Dyn(X)=E_6, E_7, E_8, A_1+E_7$, or $A_2+E_6$.
When $\Dyn(X) \neq A_1+E_7$ or $A_2+E_6$ in addition, analysis similar to that in the proof of Proposition \ref{Artin2} gives the defining equation of $X$ in $\PP$ as in Table \ref{table:Artin3}.

   \begin{table}[ht]
\caption{Artin coindices and defining equations of Du Val del Pezzo surfaces of Dynkin type $E_6$, $E_7$, or $E_8$ in $p=3$}
 \begin{tabular}{|c|c|l|c|} \hline
          $K_Y^2$       &    $Y$       &\multicolumn{1}{c|}{defining equation of $X \subset \PP$}& $\Dyn'(X)$ \\ \hline \hline
 \multirow{2}{*}{$3$}   &    $Y^{\textup{(1)}}_3(E_6)$& $wz^2+x^3+y^2z       $& $E_6^0$\\ \cline{2-4}
                        &    $Y^{\textup{IV}}_3(E_6)$ & $wz^2+x^3+y^2z+x^2z  $& $E_6^1$\\ \hline\hline
 \multirow{2}{*}{$2$}   &    $Y^{\textup{(1)}}_2(E_7)$    & $w^2+yz^3+xy^3        $& $E_7^0$\\ \cline{2-4}
                        &    $Y^{\textup{IV}}_2(E_7)$   & $w^2+yz^3+xy^3+y^2z^2  $& $E_7^1$\\ \hline \hline 
 \multirow{3}{*}{$1$}   &    $Y^{\textup{(1)}}_1(E_8)$                 & $w^2+z^3+xy^5       $& $E_8^0$\\ \cline{2-4}
                        &    $Y^{\textup{I}}_1(E_8)$                  & $w^2+z^3+xy^5+y^4z  $& $E_8^1$\\ \cline{2-4}
                        &    $Y^{\textup{IV}}_1(E_8)$                  & $w^2+z^3+xy^5+y^2z^2$& $E_8^2$\\ \hline
 \end{tabular}
 \label{table:Artin3}
\end{table}

Next, suppose that $\Dyn(X)=A_1+E_7$ or $A_2+E_6$.
Then a suitable choice of blow-down $Y \to Y'$ gives the minimal resolution $Y'$ of a del Pezzo surface of Picard rank one of type $E_7$ or $E_6$ as in Table \ref{Artin3-1'}. 

    \begin{table}[ht]
\caption{Artin coindices of Du Val del Pezzo surfaces of Dynkin type $A_1+E_7$ or $A_2+E_6$ in $p=3$}
  \begin{tabular}{|c|c|c|} \hline
         $Y$ &$Y'$& $\Dyn'(X)$ \\ \hline \hline
         $Y^{\textup{V}}_1(A_1+E_7)$  & $Y^{\textup{(1)}}_2(E_7)$& $A_1+E_7^0$\\ \hline
         $Y^{\textup{XI}}_1(A_1+E_7)$ & $Y^{\textup{IV}}_2(E_7)$& $A_1+E_7^1$\\ \hline
         $Y^{\textup{(2)}}_1(A_2+E_6)$& $Y^{\textup{(1)}}_3(E_6)$& $A_2+E_6^0$\\ \hline
         $Y^{\textup{III}}_1(A_2+E_6)$& $Y^{\textup{IV}}_3(E_6)$& $A_2+E_6^1$\\ \hline
  \end{tabular}
  \label{Artin3-1'}
\end{table}

Combining these results, we get the assertion.
\end{proof}

We need the following proposition to prove Theorem \ref{F split Intro}.

\begin{prop}\label{defeqNBdeg1char3}
There are coordinates $[x: y: z: w]$ of $\PP_k(1,1,2,3)$ such that the defining equation of the Du Val del Pezzo surface $X(4A_2)$ is 
\begin{align}\label{eq:4A2}
w^2+z^3-x^2y^2(x+y)^2=0.   
\end{align}
\end{prop}

\begin{proof}
Take $Y$ as the minimal resolution of $X$.
Fix coordinates $[s: t: u]$ of $\PP^2_k$.
Then there is $h \colon Y \to \PP^2_k$ such that $h$ is a blowing-up along all the $\FF_3$-rational points on $\{u \neq 0\}$ except $[0:0:1]$ by \cite[Proposition 5.1 (5)]{KN}.
Set 
\begin{align}\label{eq:4A_2coord}
\begin{split}
x \coloneqq& s(s+u)(s-u), \ \ y \coloneqq t(t+u)(t-u), \\
z \coloneqq& (s+u)(s-u)(t+u)(t-u)(s+t+u)(s+t-u),\text{ and }\\
w \coloneqq& u(s+u)(s-u)(t+u)(t-u)(s+t+u)(s+t-u)(s-t+u)(s-t-u).
\end{split}
\end{align}
Then $x, y \in H^0(\PP^2_k, h_*\sO_Y(-K_Y))$, $z \in H^0(\PP^2_k, h_*\sO_Y(-2K_Y))$ and $w \in H^0(\PP^2_k, h_*\sO_Y(-3K_Y))$ 
because $H^0(\PP^2_k, h_*\sO_Y(-nK_Y)) \subset H^0(\PP^2_k, \sO_{\PP^2_k}(3n))$ consists of functions which has zero of order at least $n$ at each points in $h(E_h)$ for $n \geq 1$.
Moreover, it is easy to check that $\{x^3, x^2y, xy^2, y^3, xz, yz, w\}$ is a basis of $H^0(\PP^2_k, h_*\sO_Y(-3K_Y))$.
Hence $X$ is the closure of the image of the rational map
\begin{align*}
    \Phi \colon \PP^2_k &\dashrightarrow \PP_k(1,1,2,3)\\ 
    [s:t:u] & \mapsto [x:y:z:w].
\end{align*}
Substituting (\ref{eq:4A_2coord}) into the left-hand side of (\ref{eq:4A2}) yields zero as a polynomial of $s,t$, and $u$.
Therefore, (\ref{eq:4A2}) is a defining equation of $X(4A_2)$.  
\end{proof}

\subsection{Proof of Theorem \ref{isom}}
In this subsection, we prove Theorem \ref{isom}.


\begin{proof}[Proof of Theorem \ref{isom}]
When $p=0$ (resp.\ $p > 3$), the assertion (1) has already proven by \cite[Theorem 1.2]{Ye} (resp.\ \cite[Theorem B.7]{Lac}).
Thus, we assume that $p=2$ or $3$.

First, we assume that $p=2$.
In this case, the assertion (1) follows from Tables \ref{table:deg1}--\ref{table:deg6} and Lemmas \ref{deg2isom}, \ref{deg2isom-2}, and \ref{deg3isom}--\ref{A_4}.
The assertion (2) follows from Proposition \ref{Artin2}.

Next, we assume that $p=3$.
In this case, the assertion (1) follows from Tables \ref{table:deg1'}--\ref{table:deg6'} and Lemmas \ref{deg2isom'}--\ref{D_5'}.
The assertion (2) follows from Proposition \ref{Artin3}.
\end{proof}

\section{Proof of Theorem \ref{F split Intro}}\label{sec:app}

In this section, we prove Theorem \ref{F split Intro}. 
We also give a corollary of this theorem. 

\begin{prop}\label{ordinary member}
Let $X$ be a Du Val del Pezzo surface. Suppose that $p=2$ and $X$ is $F$-split.  
Then a general anti-canonical member of $X$ is an ordinary elliptic curve.
\end{prop}
\begin{proof}
We first recall that an elliptic curve is ordinary if and only if it is $F$-split (Remark \ref{remF-split} (3)).
By replacing $X$ with its minimal resolution, we may assume that $X$ is a smooth weak del Pezzo surface.
Then we can take a birational morphism $X\to \PP^2_{[x:y:z]}$.
Note that a general member of $|-K_X|$ is isomorphic to its image on $\PP^2_{[x:y:z]}$.
By Fedder's criterion, a cubic curve in $\PP^2_{[x:y:z]}$ is $F$-split if and only if the coefficient of $xyz$ is not equal to zero.
In particular, the locus consisting of $F$-split members in $|-K_X|$ is open.
Thus, it suffices to find an anti-canonical member which is smooth and $F$-split. 

Since $X$ is $F$-split, there exists a section \[\sigma\in \Hom_{\sO_X}(F_{*}\sO_X, \sO_X)\cong H^0(X, \sO_X(-K_X))\] of the Frobenius map $\sO_X\to F_{*}\sO_X$.
Let $\Sigma\in |-K_X|$ is a Cartier divisor corresponding to $\sigma$.
Let $C\in |-K_X|$ be a general member and $V$ the pencil generated by $C$ and $\Sigma$.
If $K_X^2 \neq 2$, then a general member of $V$ intersects $\Sigma$ transversally by Lemma \ref{basic} (2) and (6).
If $K_X^2=2$, then for general two members $C_1$ and $C_2$ of $V$, both $C_1 \cap \Sigma$ and $C_2 \cap \Sigma$ coincides with $C_1 \cap C_2$.
Moreover, one of them is smooth along $C_1 \cap C_2$ since otherwise $2=K_Y^2=(C_1 \cdot C_2) \geq 4$.
Thus, $C$ is smooth at $C\cap \Sigma$ in both cases.

Let $\phi \colon Z\to \PP^1_k$ be an elimination of a rational map associated to the pencil $V$.
Then $Z$ is normal and $F$-split by Lemma \ref{F-split;bl-up}.
Thus a general $\phi$-fiber is smooth and $F$-split by \cite[Corollary 2.3]{PZ}.
Since a general member of $V$ and $\Sigma$ intersect transversally, a general $\phi$-fiber is isomorphic to its image on $X$. 
In particular, we can take a member of $|-K_X|$ which is smooth and $F$-split, as desired.
\end{proof}

\begin{rem}\label{supersing}
Proposition \ref{ordinary member} does not hold without the assumption of characteristic.
For example, assuming $p=3$, consider a Du Val del Pezzo surface  $\{w^2+z^3+x^2y^2z-x^4z+x^6=0\}\subset \PP_k(1,1,2,3)_{[x:y:z:w]}$.
Then we can see that it is $F$-split, but a general anti-canonical member is $\{w^2+z^3+a^2x^4z-x^4z+x^6=0\}$ with $a \in k \setminus \{\pm 1\}$, which is a supersingular elliptic curve.
\end{rem}

Now we prove Theorem \ref{F split Intro}.

\begin{proof}[Proof of Theorem \ref{F split Intro}]
By Lemma \ref{basic} (1), we may assume that $p=2$ or $3$.
When $p=2$, the assertion follows from Proposition \ref{ordinary member}.

Now suppose that $p=3$. 
Let $Y$ be the minimal resolution of $X$.
If $Y \cong Y^{\textup{(1)}}_1(E_8)$ (resp.~$Y^{\textup{(3)}}_1(4A_2)$) as in Table \ref{table:deg1'}, then $X$ is not $F$-split by Table \ref{table:Artin3} (resp.~Proposition \ref{defeqNBdeg1char3}) and Lemma \ref{Fedder}, a contradiction.
Hence it suffices to show that $Y$ is not isomorphic to $Y^{\textup{(2)}}_1(A_2+E_6)$ by \cite[Theorem 1.4 and Proposition 4.1]{KN}.
On one hand, Table \ref{table:Artin3} and Lemma \ref{Fedder} show that 
$Y^{\textup{(1)}}_3(E_6)$ is not $F$-split.
Combining Table \ref{Artin3-1'} and Remark \ref{remF-split} (1), we conclude that $Y^{\textup{(2)}}_1(A_2+E_6)$ is also not $F$-split.
On the other hand, $Y$ is $F$-split by Remark \ref{remF-split} (2). 
Hence $Y \not \cong Y^{\textup{(2)}}_1(A_2+E_6)$.
\end{proof}

At the end of this paper, let us state a corollary of Theorem \ref{F split Intro}.
The second cohomology of the tangent bundle is important because this contains local-to-global obstructions to deformations (cf.~\cite[Theorem 4.13]{LN}). 
In characteristic $p=2$ or $3$, there exists a Du Val del Pezzo surface $X$ such that $H^2(X, T_X)\neq 0$ by \cite[Remark 1.8]{KN}. On the other hand, if $X$ is $F$-split, then Theorem \ref{F split Intro} shows the following.
Recall that we say that a normal projective surface $S$ has \textit{negative Kodaira dimension} if $\kappa(T, K_T)=-\infty$ for some resolution $T\to S$. 

\begin{cor}\label{tangent}
Let $S$ be a normal projective surface with only Du Val singularities of negative Kodaira dimension.
Suppose that $S$ is $F$-split. Then $H^2(S, T_S)=0$.
\end{cor}
\begin{proof}
By running a $K_S$-MMP, we obtain a birational contraction $\phi\colon S\to S'$ and a Mori fiber space $S'\to B$. 
Note that $K_S$ is not pseudo-effective because $S$ has only Du Val singularities and negative Kodaira dimension.
Since 
\[\phi_{*}(\Omega_S^{[1]}\otimes \sO_S(K_S))\hookrightarrow (\phi_{*}(\Omega_S^{[1]}\otimes \sO_S(K_S)))^{**}=\Omega_{S'}^{[1]}\otimes \sO_{S'}(K_{S'}),
\] the Serre duality yields
\[
H^2(S, T_S)\cong H^0(S, \Omega_S^{[1]} \otimes \sO_S(K_S))\subset H^0(S, \Omega_{S'}^{[1]}\otimes \sO_{S'}(K_{S'})),
\] 
where $\Omega_S^{[1]}$ denotes the reflexive hull of $\Omega_S$.
When $\dim \, B=1$, then the assertion follows from \cite[Theorem 5.3 (1)]{Kaw2}. 
Now we assume that $\dim B=0$. Since $S'$ is $F$-split
by Remark \ref{remF-split} (1), it follows that a general member of $|-K_{S'}|$ is smooth by Theorem \ref{F split Intro}. Hence the proof similar to \cite[Proposition 3.1]{KN} shows that $H^0(S, \Omega_{S'}^{[1]}\otimes \sO_{S'}(K_{S'}))=0$.
\end{proof}


\section*{Acknowledgements}
The authors would like to thank Shou Yoshikawa for useful discussion.
They also wish to express their gratitude to Professor Keiji Oguiso and Professor Shunsuke Takagi for helpful advice.
The first author was supported by JSPS KAKENHI Grant Number JP19J21085 and JP22J00272.
The second author was supported by JSPS KAKENHI Grant Number JP19J14397.


\end{document}